\documentclass[12pt]{article}
\usepackage[margin = 20mm]{geometry}
\usepackage{amsmath}
\usepackage{amsfonts}
\usepackage{amssymb}
\usepackage{amsthm}
\usepackage{mathrsfs}
\usepackage{makecell}
\usepackage{cases}
\usepackage{enumerate}
\usepackage{tikz}
\usepackage{array,color,colortbl}   
\providecommand{\mk}{\cellcolor{blue!15}}
\usepackage{abstract}
\usepackage{caption}
\usepackage[backend = biber, doi = false, url = false, isbn = false]{biblatex}
\renewbibmacro{in:}{}
\newtheorem{thm}{Theorem}[section]
\newtheorem{cor}[thm]{Corollary}
\newtheorem{lem}[thm]{Lemma}

\theoremstyle{definition}

\newtheoremstyle{break}
{\topsep}{\topsep}%
{}{}%
{\bfseries}{}%
{\newline}{}%
\theoremstyle{break}

\def\eref#1{$(\ref{#1})$}
\def\sref#1{\S$\ref{#1}$}
\def\lref#1{Lemma~$\ref{#1}$}
\def\tref#1{Theorem~$\ref{#1}$}
\def\cyref#1{Corollary~$\ref{#1}$}

\renewcommand{\=}{\equiv}
\renewcommand{\geq}{\geqslant}
\renewcommand{\leq}{\leqslant}
\renewcommand{\ge}{\geqslant}
\renewcommand{\le}{\leqslant}
\renewcommand{\emptyset}{\varnothing}
\renewcommand{\l}{\langle}
\renewcommand{\r}{\rangle}

\renewcommand{\P}{\mathcal{P}}

\def\PP{\mathbb{P}}

\def\Id{\epsilon}
\def\D{\mathcal{D}}
\def\DD{\mathbb{D}}

\def\C{\mathcal{C}}
\def\EE{\mathbb{E}}
\def\II{\mathbb{I}}
\def\E{\mathcal{E}}
\def\Z{\mathbb{Z}}
\def\K{\mathcal{K}}
\def\inc{\hookrightarrow}

\def\Der{\textnormal{Der}}
\def\fix{\textnormal{fix}}
\def\supp{\textnormal{supp}}
\def\Sym{\textnormal{Sym}}
\makeatletter
\g@addto@macro\bfseries{\boldmath}
\makeatother

\title{Commuting Pairs in Quasigroups}
\author{Jack Allsop \ \ Ian M. Wanless\\
	\small School of Mathematics\\[-0.5ex]
	\small Monash University\\[-0.5ex]
	\small Vic 3800, Australia\\
	\small\tt jack.allsop@monash.edu \ \ ian.wanless@monash.edu}
\date{}

\begin{document}
	
	\maketitle
	
	\begin{abstract}
		A quasigroup is a pair $(Q, *)$ where $Q$ is a non-empty set and $*$
		is a binary operation on $Q$ such that for every $(a, b) \in Q^2$
		there exists a unique $(x, y) \in Q^2$ such that $a*x=b=y*a$. Let
		$(Q, *)$ be a quasigroup. A pair $(x, y) \in Q^2$ is a commuting
		pair of $(Q, *)$ if $x * y = y * x$. Recently, it has been shown
		that every rational number in the interval $(0, 1]$ can be attained
		as the proportion of ordered pairs that are commuting in some
		quasigroup. For every positive integer $n$ we establish the set of
		all integers $k$ such that there is a quasigroup of order $n$ with
		exactly $k$ commuting pairs. This allows us to determine, for a
		given rational $q \in (0, 1]$, the spectrum of positive integers $n$
		for which there is a quasigroup of order $n$ whose proportion of
		commuting pairs is equal to $q$.
		
		\medskip
		
	\end{abstract}
	
	\section{Introduction}
	
	A \emph{partial quasigroup} is a triple $(Q, P, *)$ where $Q$ is a
	set, $P \subseteq Q^2$ and $* : P \to Q$ is a map such that for each
	$(a, b) \in Q^2$ there exists at most one $x \in Q$ such that $a*x=b$
	and at most one $y \in Q$ such that $y*a=b$. A \emph{quasigroup} is a
	partial quasigroup of the form $(Q,Q^2,*)$. We will generally refer to
	the partial quasigroup $(Q, P, *)$ simply by $(Q,*)$ or by $Q$.
	The \emph{order} of $Q$ is $|Q|$.
	If $(x,y) \in P$ and $(y,x) \in P$ then
	$(x, y)$ is a \emph{commuting pair of $Q$} provided $x*y = y*x$,
	and otherwise $(x,y)$ is \emph{non-commuting}.
	Define
	\begin{equation*}\label{e:commpairs}
		\C(Q) = \big|\{(x, y) \in P : x * y = y * x\}\big|
	\end{equation*}
	to be the number of (ordered) pairs of $Q$ which commute. Suppose that $P$ has the extra property that $(x, y) \in P$ if and only if $(y, x) \in P$, for every $(x, y) \in Q^2$. Then $Q$ is \emph{commutative} if $\C(Q) = |P|$ and $Q$ is \emph{anti-commutative} if the only commuting pairs of $Q$ are those of the form $(x, x) \in P$. Assume now that $P=Q^2$ so that $Q$ is a quasigroup. It is simple to see that $\C(Q) \geq |Q|$ with equality if and only if $Q$ is anti-commutative, and $\C(Q) \leq |Q|^2$ with equality if and only if $Q$ is commutative.
	Cyclic groups provide examples of commutative quasigroups for all
	positive integer orders.  Anti-commutative quasigroups are also not
	hard to build, giving us the following folklore result.
	
	\begin{thm}\label{t:anticomm}
		There exists an commutative quasigroup of order $n$ for all positive
		integers $n$. There exists an anti-commutative quasigroup of order
		$n$ for all positive integers $n \neq 2$.
	\end{thm}

	Let $Q$ be a quasigroup of order $n$ and define 
	\[
	\P(Q) = \frac{\C(Q)}{n^2}
	\]
	to be the proportion of (ordered) pairs in $Q^2$ which commute. Hall~\cite{Hall} introduced the property $\P(Q)$ in the case where $Q$ is a finite group. Since then many authors, including Erd\H{o}s and Tur\'{a}n~\cite{erdtur}, have studied this same quantity and many interesting results have been proved (see, e.g.,~\cite{commlb, commprob, pgupp}). There have also been some natural generalisations of the proportion of commuting pairs to infinite groups (see, e.g.,~\cite{commcomp1, comminf}), rings (see, e.g.,~\cite{commrings2, commrings1}) and semigroups~\cite{semidense, semiMacHale, semirat, semisingle}. See~\cite{commhis} for a brief history on the study of commuting pairs in algebraic objects. Recently, Lycan~\cite{quasrat} showed that $\{\P(Q) : Q \text{ is a finite quasigroup}\} = (0, 1] \cap \mathbb{Q}$. That is, the proportion of commuting pairs in a quasigroup can achieve any rational number in the interval $(0, 1]$. The analogous result for semigroups is also known to hold~\cite{semirat}. However, for groups $\P(Q)$ is much more restricted. A well known theorem of Gustafson~\cite{pgupp} says that if $Q$ is a non-abelian group then $\P(Q) \leq 5/8$. Lycan's result shows that no analogous result to Gustafson's holds for quasigroups. The situation for more structured varieties of quasigroups, such as Moufang loops, remains open. It will not be addressed in the present work, although it seems a worthy question for future research
	
	Let $n$ be a positive integer and define
	\[
	\C(n) = \bigcup_{Q} \C(Q),
	\]
	where the union runs over all quasigroups $Q$ of order $n$. Also define
	\[\D(n) = \{n, n+2, n+4, \ldots, n^2-6\} \cup \{n^2\}.\]
	It is not hard to see that $\C(n) \subseteq \D(n)$ for every positive integer $n$ (see \lref{l:commvalid}). If $\C(n) = \D(n)$ then we say that $n$ is \emph{saturated}. We prove the following theorem.
	
	\begin{thm}\label{t:commpairs}
		Let $n$ be a positive integer. If $n=4$ then $\C(n) = \D(n) \setminus \{10\}$, if $n=5$ then $\C(n) = \D(n) \setminus \{17\}$, and if $n \not\in \{4, 5\}$ then $n$ is saturated.
	\end{thm}
	
	\tref{t:commpairs} says that for any positive integer $n$, and any integer $k \in \D(n)$, there is a quasigroup of order $n$ with exactly $k$ commuting pairs, unless $(n, k) \in \{(4, 10), (5, 17)\}$.
	Let $q \in (0, 1] \cap \mathbb{Q}$ and define 
	$\K(q)$ to be the set of positive integers $n$ for which
	there is a quasigroup $Q$ of order $n$ with $\P(Q)=q$.
	Lycan~\cite{quasrat} suggested the problem of determining $\min(\K(q))$ for each $q \in (0, 1] \cap \mathbb{Q}$. Note that $\K(1)$ is the set of positive integers since commutative quasigroups exist for every order. Our second main theorem gives a complete description of $\K(q)$ for each $q \in (0, 1) \cap \mathbb{Q}$.
	
	\begin{thm}\label{t:Kq}
		Let $q=a/b \in (0, 1) \cap \mathbb{Q}$ where $a$ and $b$ are coprime, and let $k$ be the smallest positive integer such that $kb$ is a square. Define 
		\[
		\mathcal{S} = \Big\{x(kb)^{1/2} : x \in \Z \text{ and } x \geq \max\left\{\left\lceil b^{1/2}/(ak^{1/2}) \right\rceil, \left\lceil 6^{1/2}/(kb-ka)^{1/2}\right\rceil\right\}\Big\},
		\]
		if $a \equiv b \bmod 2$ or $k$ is even, and define 
		\[
		\mathcal{S} = \Big\{x(kb)^{1/2} : x \in \Z \text{ is even and } x \geq \max\left\{\left\lceil b^{1/2}/(ak^{1/2}) \right\rceil, \left\lceil6^{1/2}/(kb-ka)^{1/2}\right\rceil\right\}\Big\}
		\] 
		otherwise. Then $\K(q) = \mathcal{S}$ up to these exceptions:
		\begin{itemize}
			\item If $q = 5/8$ then $\K(q) = \mathcal{S} \setminus \{4\}$,
			\item If $q = 17/25$ then $\K(q) = \mathcal{S} \setminus \{5\}$.
		\end{itemize}
	\end{thm}
	
	The solution to Lycan's problem follows immediately from \tref{t:Kq}.
	
	The structure of this paper is as follows. In \sref{s:LS} we give some necessary background information on quasigroups, and their combinatorial equivalents, Latin squares. In \sref{s:prelim} we prove some preliminary results regarding quasigroups which help us to prove \tref{t:commpairs} and \tref{t:Kq}, which we do in \sref{s:main}.
	
	\section{Background}\label{s:LS}
	
	Let $Z$ be a set. The set of all permutations of $Z$ is denoted by $\Sym(Z)$. The set of fixed points of a permutation $\alpha \in \Sym(Z)$ is denoted by $\fix(\alpha)$ and $\supp(\alpha) = Z \setminus \fix(\alpha)$ is the set of points moved by $\alpha$. If $\fix(\alpha)=\emptyset$ then $\alpha$ is a \emph{derangement} and the set of all derangements of $Z$ is denoted by $\Der(Z)$.
	
	Let $(Q, *)$ be a quasigroup.
	The opposite quasigroup of $Q$ is the quasigroup $(Q, \circ)$ defined by $i \circ j = j*i$. Let $W \subseteq Q$. If $(W, *|_W)$ is itself a quasigroup then it is a \emph{subquasigroup} of $Q$.
	
	Let $(Q, P, *)$ and $(Q', P', \circ)$ be partial quasigroups. If there exists bijective maps $\alpha, \beta, \gamma : Q \to Q'$ such that $P' = \{(\alpha(x), \beta(y)) : (x, y) \in P\}$ and $\alpha(x) \circ \beta(y) = \gamma(x*y)$ for every $(x, y) \in P$ then $Q$ and $Q'$ are \emph{isotopic} and $(\alpha, \beta, \gamma)$ is an \emph{isotopism} from $Q$ to $Q'$. If $(Q, *) = (Q', \circ)$ then $(\alpha, \beta, \gamma)$ is an \emph{autotopism} of $Q$. If $\alpha = \beta = \gamma = \psi$ for some $\psi \in \Sym(Q)$ then $\psi$ is an \emph{automorphism} of $Q$. We will be particularly interested in isotopic partial quasigroups $Q$ and $Q'$ where $(\alpha, \Id, \Id)$ is an isotopism from $Q$ to $Q'$. In this case we denote $Q'$ by $Q\l\alpha\r$.
	
	Let $n$ be a positive integer. A \emph{Latin square} of order $n$ is an $n \times n$ matrix of $n$ symbols, each of which occur exactly once in each row and column. We will always assume that the rows and columns of a Latin square are indexed by its symbol set. This allows us to think of Latin squares and finite quasigroups as equivalent objects, in the sense that the Cayley table of a quasigroup is a Latin square and every Latin square is the Cayley table of some finite quasigroup. Let $L$ be a Latin square with symbol set $S$. A pair $(i, j) \in S^2$ is a commuting pair of $L$ if $L_{ij} = L_{ji}$. If all pairs in $S^2$ are commuting pairs of $L$ then $L$ is \emph{symmetric} and if the only pairs in $S^2$ which are commuting pairs of $L$ are those of the form $(s, s)$ then $L$ is \emph{anti-symmetric}. Symmetric Latin squares are equivalent to commutative quasigroups and anti-symmetric Latin squares are equivalent to anti-commutative quasigroups.  
	
	Let $(Q, P, *)$ be a partial quasigroup and suppose that there is some set $H \subseteq Q$ such that $P = Q^2 \setminus H^2$ and $i*j \in H$ implies that $(i, j) \in (Q \setminus H)^2$. Then $Q$ is a \emph{partial quasigroup with hole $H$}. For positive integers $n$ and $m$ we define the set of all partial quasigroups of order $n$ with a hole of order $m$ by $\Omega(n, m)$.
	
	Let $Q \in \Omega(n, m)$ have hole $H$. Let $Q' = (H, \circ)$ be a quasigroup. We can define a quasigroup $(Q' \inc Q, \cdot)$ by
	\[
	i \cdot j = \begin{cases}
		i*j & \text{if } (i, j) \in Q^2 \setminus H^2, \\
		i \circ j & \text{if } (i, j) \in H^2. 
	\end{cases}
	\]
	We say that $Q' \inc Q$ has been built by pasting $Q'$ into the hole of $Q$.
	
	The following theorem, which is a direct consequence of work of Cruse~\cite{symmhole}, proves existence of many commutative quasigroups with holes.
	
	\begin{thm}\label{t:symmholeeven}
		Let $n$ and $m \leq \lfloor n/2 \rfloor$ be positive integers. Also assume that $m$ is odd if $n$ is odd. There exists a commutative member of $\Omega(n, m)$.
	\end{thm}
	
	Two quasigroups $(Q, *)$ and $(Q, \circ)$ are \emph{orthogonal} if $\big|\{(i*j, i \circ j) : (i, j) \in Q^2\}\big| = |Q|^2$. Equivalently, two quasigroups are orthogonal if their Cayley tables are orthogonal as Latin squares.
	Let $(Q, *)$ be a quasigroup. If $Q$ is orthogonal to its opposite quasigroup then $Q$ is \emph{self-orthogonal}. We are interested in self-orthogonal quasigroups because of the following simple observation, which is part of the folklore of the subject. 
	
	\begin{lem}\label{l:solsantisymm}
		Any self-orthogonal quasigroup is anti-commutative.
	\end{lem}
	
	The converse to \lref{l:solsantisymm} does not hold. For example, the anti-symmetric Latin squares in \eref{e:antisymm36} are not self-orthogonal.	
	\begin{equation}\label{e:antisymm36}
		\begin{minipage}{.45\textwidth}
			\begin{equation*}
				\begin{pmatrix}
					0&1&2\\
					2&0&1\\
					1&2&0
				\end{pmatrix}
			\end{equation*}
		\end{minipage}
		\begin{minipage}{.45\textwidth}
			\begin{equation*}
				\begin{pmatrix}
					2&1&4&0&5&3\\
					0&2&5&2&4&1\\
					1&4&3&5&2&0\\
					4&5&1&3&0&2\\
					3&2&0&4&1&5\\
					5&0&2&1&3&4\\
				\end{pmatrix}
			\end{equation*}
		\end{minipage}
	\end{equation}
	
	Brayton, Coppersmith and Hoffman~\cite{sols} showed that a self-orthogonal quasigroup of order $n$ exists for all positive integers $n \not\in \{2, 3, 6\}$. Combining this fact with \lref{l:solsantisymm} and \eref{e:antisymm36} is one way to prove \tref{t:anticomm}.
	
	Let $n$ and $m$ be positive integers and let $Q \in \Omega(n, m)$ have hole $H$. Then $Q$ is self-orthogonal if $\big|\{(i*j, j*i) : (i, j) \in Q^2 \setminus H^2\}\big| = |Q^2 \setminus H^2|$.
	The following theorem~\cite{solshole5} provides us with an existence result for self-orthogonal quasigroups with holes.
	
	\begin{thm}\label{t:solshole}
		Let $n$ and $m \leq \lfloor (n-1)/3 \rfloor$ be positive integers with
		\begin{equation}\label{e:nmpairs}
			(n, m) \not\in \big\{(6, 1), (8, 2), (20, 6), (26, 8), (32, 10)\big\}.
		\end{equation}
		There exists a self-orthogonal member of $\Omega(n, m)$.
	\end{thm}
	
	The exceptions $(n, m) \in \{(6, 1), (8, 2)\}$ in \tref{t:solshole} are the only exceptions known to be genuine. Combining \tref{t:solshole} with \lref{l:solsantisymm} we obtain the following.
	
	\begin{cor}\label{c:antisymmhole}
		Let $n$ and $m \leq \lfloor (n-1)/3 \rfloor$ be positive integers satisfying \eref{e:nmpairs}. There exists an anti-commutative member of $\Omega(n, m)$.
	\end{cor}
	
	In \sref{s:prelim} we extend \cyref{c:antisymmhole} by proving the
	following theorem.
	
	\begin{thm}\label{t:antisymmholenew}
		Let $n \geq 3$ and $m \leq \lfloor n/2 \rfloor$ be
		positive integers, with $(n, m) \neq (4, 2)$.  There
		exists an anti-commutative member of $\Omega(n, m)$.
	\end{thm}
	
	Let $L$ be a Latin square with symbol set $S$ of cardinality $n$, and let $i$ and $j$ be distinct elements of $S$. The permutation mapping row $i$ of $L$ to row $j$ of $L$, denoted by $r_{ij}$, is defined by $r_{ij}(L_{ik}) = L_{jk}$ for every $k \in S$. Such permutations are called \emph{row permutations} of $L$. Let $\rho$ be a cycle in $r_{ij}$ and in row $i$ (or row $j$) let the set of columns containing the symbols involved in $\rho$ be $C$. The set of entries whose cells are in $\{i, j\} \times C$ is called a \emph{row cycle} of $L$. The \emph{length} of this row cycle is $|C|$. This row cycle is uniquely determined by the rows $\{i, j\}$ and a single column in $C$. Thus we denote this row cycle by $\rho(i, j, c)$ for any $c \in C$.
	Row cycles give us a way to create a new Latin square by perturbing an old one, in a method known as cycle
	switching \cite{cycswitch}. 
	Let $\rho(i, j, c)$ be a row cycle of $L$ which involves the entries in cells $\{i, j\} \times C$ for some $C \subseteq S$. A new Latin square $L'$ can be defined by
	\[
	L'_{xy} = \begin{cases}
		L_{iy} & \text{if } x=j \text{ and } y \in C, \\
		L_{jy} & \text{if } x=i \text{ and } y \in C, \\
		L_{xy} & \text{otherwise}.
	\end{cases}
	\]
	We will say that $L'$ has been obtained from $L$ by \emph{switching}
	on $\rho(i, j, c)$. 
	
	Let $n$ be a positive integer. If $n=4$ then define $\mathbb{D}(n) = \{4, 6, 8, 16\}$, if $n=5$ then define $\mathbb{D}(n) = \{5, 7, 9, 11, 13, 15, 19, 25\}$, otherwise define $\mathbb{D}(n) = \D(n)$. To prove \tref{t:commpairs} we must construct a quasigroup, or Latin square, of order $n$ with exactly $k$ commuting pairs for each positive integer $n$ and each $k \in \DD(n)$. To do this we will use two different constructions. For the first construction, we paste a quasigroup into a partial quasigroup with a hole and a known number of commuting pairs. For the second construction, we begin with a Latin square with a known number of commuting pairs and use cycle switching to create Latin squares with different numbers of commuting pairs.	
	
	\section{Preliminary results}\label{s:prelim}
	
	In this section we construct some partial quasigroups that will be
	useful later. For a positive integer $n$ let $\Z_n$ denote the set $\{0, 1, \ldots, n-1\}$ under addition modulo $n$. For convenience, if $x \in \Z_n$ and $y \not\in \Z_n$ then we set $x+y=y$. Let $m$ and $n$ be integers satisfying $\lceil
	n/3\rceil\leq m < n/2$, and let $k=n-m$. Define $Q=\Z_k\cup F$ where $F$ is a set
	satisfying $|F|=m$ and $F\cap\Z_k=\emptyset$. We will define
	quasigroups on $Q$ which have the $k$-cycle
	$\psi=(0,1,\dots,k-1)\in\Sym(Q)$ as an automorphism.  The structure of
	such quasigroups is well understood (see, for example,
	\cite{diagcyc}). They all contain a subquasigroup on $F$. Since this
	subquasigroup can easily be replaced by any other subquasigroup on
	$F$, we will leave it as a hole, to be filled in later.  We call the
	remainder (outside the hole) a \emph{diagonally cyclic partial
		quasigroup} (DCPQ).  
	
	A \emph{partial orthomorphism of $\Z_k$ with deficit $d$} is an
	injective map $\theta:S\rightarrow\Z_k$ where $S\subseteq\Z_k$ and
	$|S|=k-d$, for which the map $\phi:S\rightarrow\Z_k$
	defined by $\phi(x)=\theta(x)-x$ is also injective. As noted in \cite{diagcyc},
	in order to define a DCPQ it suffices to define $0*x$ for $x\in Q$ and
	$x*0$ for $x\in F$ satisfying the following conditions:
	\begin{enumerate}[\text{Condition} $1$:]
		\item Let $S=\{x\in\Z_k:0*x\in\Z_k\}$.
		The map $\theta:S\rightarrow\Z_k$ defined by $x\mapsto0*x$
		is a partial orthomorphism of $\Z_k$ with deficit $d=m$.
		\item $\{0*x:x\in F\}=\Z_k\setminus\{0*x:x\in S\}$.
		\item $\{x*0:x\in F\}=\Z_k\setminus\{0*y-y:y\in S\}$.
	\end{enumerate}
	
	We now define a family of DCPQs on $Q$. Let $F = \{f_0, f_1, \ldots, f_{m-1}\}$. Define a map $g : Q \to Q$ as follows. If $x \in \Z_k$ then 
	\begin{equation}\label{e:defg}
		g(x) = \begin{cases}
			f_x & \text{if } x \in \{0, 1, \ldots, \lceil k/2 \rceil-1\}, \\
			\lceil k/2 \rceil -x-1 & \text{if } x \in \{\lceil k/2 \rceil, \lceil k/2 \rceil+1, \ldots, \lceil k/2 \rceil+k-m-1\},\\
			f_{x-k+m} & \text{otherwise}.
		\end{cases}
	\end{equation}
	If $i \in \{0, 1, \ldots, m-1\}$ then define $g(f_i) = i$.
	Let
	\begin{equation}\label{e:AB}
		A = \{g(x)-x : x \in \Z_k\} \cap \Z_k \text{ and } B = \Z_k \setminus A. 
	\end{equation}
	Let $b = (b_0, b_1, \ldots, b_{m-1})$ be any permutation of the elements of $B$. We define a map $* = *_b : (\{0\} \times Q) \cup (F \times \{0\}) \to Q$ by $0*x = g(x)$ for every $x \in Q$ and $f_i*0 = b_i$ for every $i \in \{0, 1, \ldots, m-1\}$. It is simple to check that Conditions $1 \textendash 3$ are satisfied and so $*$ can be extended to form a DCPQ on $Q$. Denote this DCPQ by $Q(b)$.
	In the following two lemmas we count the number of commuting pairs of $Q(b)$.
	
	\begin{lem}\label{l:newcon}
		Let $n$ and $m$ be positive integers satisfying $\lceil n/3 \rceil \leq m < n/2$. Let $k=n-m$ and let $B$ be defined as in \eref{e:AB}. Let $b = (b_0, b_1, \ldots, b_{m-1})$ be a permutation of the elements of $B$. Then $\C(Q(b)) = k(2\ell+1)$ where $\ell$ is the number of solutions to $f_i*0 = 0*f_i$. 
	\end{lem}
	
	\begin{proof}
		Let $(i, j) \in \Z_k^2$ be a commuting pair of $Q = Q(b)$ with $i
		\neq j$.  Suppose first that $0*(j-i)\in F$. Since
		$j*i=i*j=0*(j-i)\in F$ it follows $0*(i-j)\in F$ and
		$0*(i-j)=0*(j-i)$, which implies that $i-j=j-i$. Therefore $k$ is
		even and $j-i=k/2$. But then $0*(k/2)=-1\notin F$, which is a
		contradiction. Therefore, $0*(j-i)\notin F$, and
		similarly $0*(i-j)\notin F$. Thus, by \eref{e:defg},
		\begin{equation}\label{e:i-jj-i}
			\{i-j, j-i\} \subseteq \{\lceil
			k/2 \rceil, \lceil k/2 \rceil+1, \ldots, \lceil k/2 \rceil+k-m-1\}.
		\end{equation}
		Now, in $\Z$ the inequality $x\ge\lceil k/2\rceil$ implies that
		$k-x\le k-\lceil k/2\rceil=\lfloor k/2\rfloor$. So the only way
		to satisfy \eref{e:i-jj-i} is if $k$ is even and $i-j=j-i=k/2$.
		However, that is impossible since it would mean that
		$$i*(i+k/2)\=(i+k/2)*i\=\psi^{k/2}(i)*\psi^{k/2}(i+k/2)\=\psi^{k/2}(i*(i+k/2))\=i*(i+k/2)+k/2$$
		mod $k$.
		It follows that the
		number of commuting pairs of $Q$ which lie in $\Z_k^2$ is $k$. It
		remains to prove that there are $\ell k$ commuting pairs of $Q$ in
		$\Z_k \times F$. Let $(i, j) \in \Z_k \times F$. Since $\psi$ is an
		automorphism of $Q$ we see that
		\[
		i*j = j*i \iff i+0*j = i+j*0 \iff 0*j=j*0.
		\]
		The lemma follows.
	\end{proof}
	
	\begin{lem}\label{l:newcons}
		Let $n$ and $m$ be positive integers with $\lceil n/3 \rceil \leq m < n/2$. Let $k=n-m$ and define $s = \max\{2m-k, m-1-\lfloor (k-2)/4 \rfloor\}$.	For each $j \in \{0, 1, \ldots, s\}$ there exists some $Q \in \Omega(n, m)$ with $C(Q) = (n-m)(2j+1)$.
	\end{lem}	
	
	\begin{proof}
		Define $A$ and $B$ as in \eref{e:AB}. We first show that $|B \cap \{0, 1, \ldots, m-1\}| = s$. Since $B = \Z_k \setminus A$ it follows that $|B \cap \{0, 1, \ldots, m-1\}| = m-|A \cap \{0, 1, \ldots, m-1\}|$. Given that $k-\lceil k/2\rceil=\lfloor k/2 \rfloor$, we have
		\[
		A = \{\lfloor k/2 \rfloor -1-2j\in\Z_k: j \in \{0, 1, \ldots, k-m-1\}\}.
		\]
		Inequalities in the remainder of the proof will be in $\Z$.
		First, we claim that $\lfloor k/2 \rfloor -1-2j \leq m$, for all $j \in \{0, 1, \ldots, k-m-1\}$. If not, then we must have $j < (\lfloor k/2 \rfloor-m-1)/2$. However, $m \geq \lceil n/3 \rceil\geq k/2$, so $(\lfloor k/2 \rfloor-m-1)/2<0$, a contradiction. Second, we claim that $\lfloor k/2 \rfloor -1-2j \geq m-k$ for all $j \in \{0, 1, \ldots, k-m-1\}$. If not then $j>(k+\lfloor k/2 \rfloor -m-1)/2$. However $(k+\lfloor k/2 \rfloor -m-1)/2>k-m-1$ since $m \geq k/2$, which gives a contradiction. It follows that if $\lfloor k/2 \rfloor -1-2j\in\{0, 1, \ldots, m-1\}$ in $\Z_k$ then $\lfloor k/2 \rfloor -1-2j\in\{0, 1, \ldots, m-1\}$ in $\Z$. Let $j' \in \{0, 1, \ldots, k-m-1\}$ be maximal such that $\lfloor k/2 \rfloor -1-2j' \geq 0$. A simple computation shows that $j' = \lfloor (\lfloor k/2 \rfloor -1)/2 \rfloor = \lfloor (k-2)/4 \rfloor$. Thus $\lfloor k/2 \rfloor -1-2j\in \{0, 1, \ldots, m-1\}$ in $\Z_k$ if and only if $j \in \{0, 1, \ldots, \min\{k-m-1, \lfloor (k-2)/4 \rfloor\}\}$. So $|B \cap \{0, 1, \ldots, m-1\}| = m-\min\{k-m, \lfloor (k-2)/4 \rfloor+1\} = \max\{2m-k, m-1-\lfloor (k-2)/4 \rfloor\}=s$, as claimed.
		
		Next we claim that $s<m$. Clearly $s \leq m$. If $s=m$ then either $k-m=0$ or $\lfloor (k-2)/4 \rfloor+1=0$. The former implies that $m=n/2$ and the latter implies that $k \leq 1$, both of which are false.
		
		Let $j \in \{0, 1, \ldots, s\}$. Note that either $j=s$ or $j\le m-2$,
		because $s<m$. Hence, we can choose a permutation
		$b=(b_0,b_1,\ldots,b_{m-1})$ of the elements of $B$ to produce exactly
		$j$ solutions to $b_i = 0*f_i$.
		The lemma now follows from \lref{l:newcon}.
	\end{proof}
	
	Our next task in this section is to prove
	\tref{t:antisymmholenew}. The first step in this direction is to prove
	that there is an anti-symmetric member of $\Omega(n, m)$ for the cases
	excluded by \eref{e:nmpairs}.
	Note that \eref{e:antisymm36} gives an anti-symmetric Latin square of
	order $6$, which deals with the exception $(n, m) = (6, 1)$. For each
	of the remaining exceptions, \lref{l:newcon} does not apply, since
	$m=(n-2)/3$.
	
	Let $n$ and $m$ be positive integers such that $(n, m) \in \{(8, 2), (20, 6), (26, 8), (32, 10)\}$, and let $k=n-m$. Let $F = \{f_0, f_1, \ldots, f_{m-1}\}$ be a set of order $m$, disjoint from $\Z_k$, and let $Q = \Z_k \cup F$. Define a function $h : Q \to Q$ as follows. If $x \in \Z_k$ then 
	\[
	h(x) = \begin{cases}
		f_x & \text{if } x \in \{0, 1, \ldots, m-1\}, \\
		1 & \text{if } x = m, \\
		m-x & \text{otherwise}. 
	\end{cases}
	\]
	If $x \in F$ then define $h(x)$ in any way such that $\{h(x) : x \in F\} = \Z_k \setminus \{h(x) : x \in \Z_k\}$. Define $h' : F \to Q$ such that $\{h'(x) : x \in F\} = \Z_k \setminus \{h(x)-x : x \in \Z_k\}$ and $h'(x) \neq h(x)$ for all $x \in F$. Define a map $* : (\{0\} \times Q) \cup \{F \times \{0\}\} \to Q$ by $0*x = h(x)$ for every $x \in Q$ and $x*0 = h'(x)$ for every $x \in F$. It is easy to verify that conditions $1\textendash3$ are satisfied, so $*$ extends to define a DCPQ on $Q$. Furthermore, by construction this DCPQ is anti-commutative. We are now ready to prove \tref{t:antisymmholenew}.
	
	\begin{proof}[Proof of \tref{t:antisymmholenew}]
		If $(n, m) \in \{(6, 1), (8, 2), (20, 6), (26, 8), (32, 10)\}$ then the result is true, as proven above. If $m \leq \lfloor (n-1)/3 \rfloor$ and \eref{e:nmpairs} holds then the result is true by \cyref{c:antisymmhole}. Otherwise if $m \neq n/2$ then the result is true by \lref{l:newcons} with $j=0$. It remains to deal with the case where $n>4$ is even and $m=n/2$.
		
		Let $n>4$ be an even integer and let $m=n/2$. Let $(Q, *)$ be an anti-commutative quasigroup of order $m$ and let $\sigma \in \Der(Q)$. Define $Q' = Q \times \{1, 2\}$ and $P =(Q')^2 \setminus (Q \times \{2\})^2$. Define a map $\circ : P \to Q'$ by
		\[
		(i, x)\circ(j, y) = \begin{cases}
			(i*j,1) & \text{if } (x, y) = (1, 1), \\
			(\sigma(i*j),2) & \text{if } (x, y) = (1, 2), \\
			(j*i,2) & \text{if } (x, y) = (2, 1).
		\end{cases}
		\]
		It is easy to see that $(Q', P, \circ) \in \Omega(n, m)$. Since $Q$ is anti-commutative, it follows that $Q'$ has no commuting pairs $((i, x), (j, y))$ with $(x, y) = (1, 1)$ and $i \neq j$. If $Q'$ had a commuting pair $((i, x), (j, y))$ with $(x, y) = (1, 2)$ then $\sigma(i*j) = i*j$, which contradicts the fact that $\sigma$ is a derangement. Hence $Q'$ is anti-commutative.
	\end{proof}
	
	For the remainder of this section we construct partial
	quasigroups with holes and a known number of commuting pairs.  Our
	method will be to permute rows in commutative partial quasigroups with
	holes.
	
	\begin{lem}\label{l:symmpermutehole}
		Let $n$ and $m \leq \lfloor n/2 \rfloor$ be positive integers. Also
		assume that $m$ is odd if $n$ is odd. For every
		$j\in\{2,3,\ldots,m\}$ there exists some $Q \in \Omega(n, m)$ with
		$\C(Q)=(n+m-2j)(n-m)$.
	\end{lem}
	
	\begin{proof}
		Let $Q' \in \Omega(n, m)$ be symmetric with hole $H$, which exists by \tref{t:symmholeeven}. Let $I = Q' \setminus H$. Let $\alpha \in \Sym(Q')$ be such that $I \subseteq \fix(\alpha)$ and $|\fix(\alpha) \cap H| = m-j$. Let $Q = Q'\l\alpha^{-1}\r$ so that $Q' = Q\l\alpha\r$. Denote the binary operation of $Q$ by $\circ$. We claim that $Q$ has $(n+m-2j)(n-m)$ commuting pairs. First, consider $(x, y) \in I \times H$ such that $x \circ y = y \circ x$. Then $\alpha(x) * y = \alpha(y) * x$. Since $x \in I$ it follows that $\alpha(x)=x$ and so $x * y = \alpha(y) * x$. Since $Q'$ is commutative it follows that $\alpha(y)=y$. Therefore the number of commuting pairs of $Q$ which lie in $I \times H$ is $|I||\fix(\alpha) \cap H| = (n-m)(m-j)$. Similarly there are $(n-m)(m-j)$ commuting pairs of $Q$ lying in $H \times I$. Also, since $I \subseteq \fix(\alpha)$ it is clear that every pair in $I^2$ is a commuting pair of $Q$. Therefore $\C(Q) = 2(n-m)(m-j)+(n-m)^2 = (n+m-2j)(n-m)$, as required.
	\end{proof}
	
	\lref{l:symmpermutehole} is proved by taking a commutative partial
	quasigroup $Q'$ with a hole $H$, finding a permutation
	$\alpha\in\Sym(Q')$ which fixes $Q'\setminus H$ and considering
	$Q'\l\alpha^{-1}\r$. Let $n$ and $m$ be positive integers and suppose
	that $Q' \in \Omega(n, m)$ is commutative with hole $H$. Let
	$\alpha\in\Sym(Q')$ fix $H$ pointwise and consider
	$Q=Q'\l\alpha^{-1}\r$. Without more information about $Q'$ we cannot
	determine $\C(Q)$ exactly. This is because determining the number of
	commuting pairs in $(\supp(\alpha))^2$ amounts to solving the equation
	$\alpha(i)*j=\alpha(j)*i$ where $(i, j) \in \supp(\alpha)^2$. However,
	the following lemma tells us that we can always choose a permutation
	$\alpha$ so that the number of such commuting pairs is small.

	For the rest of this paper set
	\[
	\beta(j) = \begin{cases}
		j^2 & \text{if } j\le2, \\
		\lfloor j(2j-3)/(j-2) \rfloor & \text{if } j \geq 3,
	\end{cases}	
	\]
	for each positive integer $j$. 
	
	\begin{lem}\label{l:dersolns}
		Let $U$ and $V$ be sets with $n = |U|$ and suppose that $* : U^2 \to V$ satisfies the following properties:
		\begin{itemize}
			\item If $u_1*u_2 = u_1*u_3$ then $u_2=u_3$,
			\item If $u_2*u_1 = u_3*u_1$ then $u_2=u_3$.
		\end{itemize} 
		There exists a derangement $\alpha \in \Der(U)$ such that the number of pairs $(x, y) \in U^2$ such that $\alpha(x) * y = \alpha(y) * x$ is at most $\beta(n)$.
	\end{lem}
	
	\begin{proof}
		The claim is trivial when $n\le2$, so we will assume that $n \geq 3$. Choose a derangement $\alpha \in \Der(U)$ uniformly at random, and let $\PP$ denote probability on the resulting probability space. Let $(x, y) \in U^2$ with $x \neq y$. Then
		\begin{align}
			\PP(\alpha(x) * y = \alpha(y) * x)
			&= \sum_{k \in U \setminus \{x\}} \PP(\alpha(x)=k \cap  k * y = \alpha(y) * x) \nonumber\\
			&= \sum_{k \in U \setminus \{x\}} \PP(\alpha(x)=k) \PP(k * y = \alpha(y) * x \mid \alpha(x)=k) \nonumber\\
			&= \frac1{n-1} \sum_{k \in U \setminus \{x\}} \PP(k * y = \alpha(y) * x \mid \alpha(x)=k).\label{e:p1}
		\end{align}
		Fix some $k \in U \setminus \{x\}$ and consider $\PP(k * y = \alpha(y) * x \mid\alpha(x)=k)$. The assumptions on $*$ imply that there is at most one $z \in U$ such that $k*y = z*x$. Hence 
		\begin{equation}\label{e:p2}
			\PP(k*y = \alpha(y) * x \mid \alpha(x)=k) = \begin{cases}
				\PP(\alpha(y)=z \mid \alpha(x)=k) & \text{if } z \in U \text{ is such that } k * y =z * k, \\
				0 & \text{if }k * y \neq z * k\text{ for all } z \in U.
			\end{cases}
		\end{equation}
		If $\alpha(x)=k$ then there are at least $n-2$ possible values for $\alpha(y)$. Hence for any $z \in U$ it is true that $\PP(\alpha(y)=z \mid \alpha(x)=k) \leq 1/(n-2)$. Combining this with \eref{e:p1} and \eref{e:p2} we obtain
		\[
		\PP(\alpha(x) * y = \alpha(y) * x) \leq \frac1{(n-2)}.
		\]
		Note that trivially, $\PP(\alpha(x) * x = \alpha(x) * x) = 1$ for every $x \in U$. Let $W$ denote the number of pairs $(x, y) \in U^2$ such that $\alpha(x) * y = \alpha(y) * x$. For a pair $(x, y) \in U^2$ let $\II_{(x, y)}$ denote the variable indicating the event that $\alpha(x) * y = \alpha(y) * x$. Let
		\[
		S = \sum_{(x, y) \in U^2} \II_{(x, y)}.
		\]
		Then, 
		\[
		\EE(S) = \sum_{(x, y) \in U^2} \PP(\II_{(x, y)}=1) = n+\frac{n(n-1)}{n-2} = \frac{n(2n-3)}{n-2}.
		\]
		The lemma follows.
	\end{proof}
	
	We can now use \lref{l:dersolns} to prove a result analogous to \lref{l:symmpermutehole} when we take a commutative partial quasigroup $Q$ with a hole $H$ and consider $Q\l\alpha^{-1}\r$ where $\alpha \in \Sym(Q)$ fixes $H$.
	
	\begin{lem}\label{l:holecomm}
		Let $n$ and $m \leq \lfloor n/2 \rfloor$ be positive integers. Also
		assume that $m$ is odd if $n$ is odd. For every
		$j\in\{2,3,\ldots,n-m\}$ there exists some $Q \in \Omega(n, m)$ with
		$\C(Q) = (n-j-m)(n-j+m)+i$, for some $i$ satisfying $i\equiv j\bmod
		2$ and $j \leq i \leq \beta(j)$.
	\end{lem}
	
	\begin{proof}
		Let $Q' \in \Omega(n, m)$ be commutative with hole $H$ of order $m$, which exists by \tref{t:symmholeeven}. Let $I = Q' \setminus H$. By \lref{l:dersolns} there exists an $\alpha \in \Sym(Q')$ such that:
		\begin{itemize}
			\item $H \subseteq \fix(\alpha)$,
			\item $|\fix(\alpha) \cap I| = n-m-j$, and
			\item there are at most $\beta(j)$ pairs $(x, y) \in \supp(\alpha)^2$ such that $\alpha(x) * y = \alpha(y) * x$.
		\end{itemize}
		Let $Q = Q'\l\alpha^{-1}\r$ and denote the binary operation on $Q$ by $\circ$.
		
		First, consider $(x, y) \in I \times H$ such that $x \circ y = y \circ x$. Then $\alpha(x) * y = \alpha(y) * x$. Since $y \in H$ it follows that $\alpha(y)=y$ and so $\alpha(x) * y = y * x$. Since $Q'$ is commutative it follows that $\alpha(x)=x$ and so the number of commuting pairs of $Q$ which lie in $I \times H$ is exactly $|\fix(\alpha) \cap I||H| = m(n-m-j)$. Similarly, the number of commuting pairs of $Q$ which lie in $H \times I$ is $m(n-m-j)$. 
		
		Now consider $(x, y) \in I \times I$ such that $x \circ y = y \circ x$, so that $\alpha(x) * y = \alpha(y) * x$. If $x \in \fix(\alpha)$ then $x * y = \alpha(y) * x$ which implies that $y \in \fix(\alpha)$ since $Q'$ is commutative. Similarly if $y \in \fix(\alpha)$ then $x \in \fix(\alpha)$ also. This gives $|\fix(\alpha) \cap I|^2 = (n-m-j)^2$ commuting pairs. Now suppose that $\{x, y\} \subseteq I \setminus \fix(\alpha)$. Then $(x, y) \in \supp(\alpha)^2$ and $\alpha(x) * y = \alpha(y) * x$. By definition of $\alpha$ there are at most $\beta(j)$ such pairs. Furthermore, the number of such pairs must be congruent to $j$ modulo $2$, since each pair $(x, x) \in \supp(\alpha)^2$ is commuting and the pair $(x, y) \in \supp(\alpha)^2$ is commuting if and only if $(y, x)$ is commuting. Therefore $\C(Q) = 2m(n-m-j)+(n-m-j)^2+i$ for some $i$ satisfying $j \leq i \leq \beta(j)$ and $i \equiv j \bmod 2$. The lemma follows.
	\end{proof}
	
	\section{Main results}\label{s:main}
	
	In this section we prove \tref{t:commpairs} and \tref{t:Kq}. 
	
	\subsection{Proof of \tref{t:commpairs}}
	
	Our first task towards proving \tref{t:commpairs} is to prove the fact that $\C(n) \subseteq \D(n)$ for all positive integers $n$.
	
	\begin{lem}\label{l:commvalid}
		Let $Q$ be a quasigroup of order $n$. Then $\C(Q) \in \D(n)$.
	\end{lem}
	
	\begin{proof}
		Let $k$ be the number of commuting pairs in $Q$. Since each pair $(x, x) \in Q^2$ is commuting and the pair $(x, y) \in Q^2$ is commuting if and only if $(y, x)$ is commuting it follows that $k \equiv n \bmod 2$ and $n \leq k \leq n^2$. It remains to show that $k \not\in \{n^2-2, n^2-4\}$.
		
		Suppose that $k = n^2-2$. Let $(x, y)$ and $(y, x)$ be the only non-commuting pairs of $Q$. Let $z_1 = x * y$ and $z_2 = y * x$. Let $z \in Q$ be such that $x * z = z_2$. Since $z_1 \neq z_2$ we cannot have $z=y$ and so $(x, z)$ is commuting. So $z * x = x * z = z_2$ which implies that $z=y$, a contradiction. 
		
		Now suppose that $k = n^2-4$. Let $(x, y)$, $(y, x)$, $(z, w)$ and $(w, z)$ be the only non-commuting pairs in $Q$. Let $u_1 = x * y$ and $u_2 = y * x$. Let $u \in Q$ be such that $x * u = u_2$ and note that $u \neq y$ since $u_1 \neq u_2$. If $\{x, u\} \neq \{z, w\}$ then $(x, u)$ is commuting and so $u * x = x * u = u_2$ which implies that $u = y$, which is false. Therefore $\{x, u\} = \{z, w\}$. Without loss of generality $w=x$. Let $u' \in Q$ be such that $y * u' = u_1$ and note that $u' \neq x$ since $u_1 \neq u_2$. Also note that $z \neq y$ since $Q$ has $n^2-4$ commuting pairs. It follows that $(u', y)$ is commuting and so $u'*y=y*u'=u_1$ which implies that $u'=x$, a contradiction.
	\end{proof}
	
	In the remainder of this subsection we work on constructing, for every positive integer $n$ and every $k \in \DD(n)$, a quasigroup, or Latin square, of order $n$ with exactly $k$ commuting pairs. For some pairs $(n, k)$ we can do this directly using cycle switching, but for most pairs we use a recursive construction based on pasting a quasigroup of some order $m \leq \lfloor n/2 \rfloor$ into a member of $\Omega(n, m)$. We therefore split this subsection into two parts: dealing with the recursive part of the construction and dealing with the base cases.
	
	\subsubsection{The recursive step}\label{ss:recursive}
	
	\tref{t:commpairs} can be stated as
	\begin{equation}\label{e:goal}
		\C(n) = \DD(n)
	\end{equation} 
	for every positive integer $n$. The goal for this section is to prove the following lemma.
	
	\begin{lem}\label{l:recursive}
		Suppose that \eref{e:goal} is true for every positive integer $n \leq 10$. Then \eref{e:goal} is true for all positive integers $n$.
	\end{lem}
	
	The following lemma lets us make use of the results proved in \sref{s:prelim}.
	
	\begin{lem}\label{l:embed}
		Let $n$ and $m$ be positive integers. If there exists some $Q \in \Omega(n, m)$ with $\C(Q) = c$ then $\{c+\ell : \ell \in \C(m)\} \subseteq \C(n)$.
	\end{lem}
	
	\begin{proof}
		Let $Q'$ be a quasigroup of order $m$ with $\C(Q') = \ell$. Then $\C(Q' \inc Q) = c+\ell$.
	\end{proof}
	
	Combining \lref{l:embed} with \tref{t:symmholeeven}, \tref{t:antisymmholenew}, \lref{l:newcons} and \lref{l:symmpermutehole} we obtain the following corollary.
	
	\begin{cor}\label{c:embedding}
		Let $n$ and $m \leq \lfloor n/2 \rfloor$ be positive integers.
		\begin{enumerate}[(i)]
			\item If $n$ is even or $m$ is odd then $\{n^2-m^2+\ell : \ell \in \C(m)\} \subseteq \C(n)$.
			\item $\{n-m+\ell : \ell \in \C(m)\} \subseteq \C(n)$, provided $n>2$ and
			$(n,m)\ne(4,2)$.
			\item If $\lceil n/3 \rceil \leq m<n/2$ then $\{(2j+1)(n-m)+\ell : j \in \{0, 1, \ldots, s\}, \ell \in \C(m)\} \subseteq \C(n)$, where $s = \max\{3m-n, m-1-\lfloor (n-m-2)/4 \rfloor\}$. 
			\item If $n$ is even or $m$ is odd then $\{(n+m-2j)(n-m)+\ell : j \in \{2, 3, \ldots, m\}, \ell \in \C(m)\} \subseteq \C(n)$.
		\end{enumerate}
	\end{cor}
	
	We can also use \lref{l:holecomm} to build quasigroups with a known
	number of commuting pairs.
	
	\begin{lem}\label{l:kconds}
		Let $n$ be a positive integer and let $m \leq \lfloor n/2 \rfloor$
		be a positive, saturated integer. Also assume that $m$ is odd if $n$
		is odd. If $k \in \D(n)$ is such that there exists some
		$j \in \{2,3, \ldots, n-m\}$ so that
		\begin{equation}\label{e:k}
			\beta(j)+m \leq k-(n-j-m)(n-j+m) \leq j+m^2-6,
		\end{equation}
		then $k \in \C(n)$.
	\end{lem}
	
	\begin{proof}
		By \lref{l:holecomm}, there exists $Q \in \Omega(n, m)$ with $\C(Q) = (n-j-m)(n-j+m)+i$, for some $i$ satisfying $j \leq i \leq \beta(j)$ and $i \equiv j \bmod 2$. Then note that $m \leq k-(n-m-j)(n+m-j)-i \leq m^2-6$ by \eref{e:k}. Furthermore, 
		\[
		k-(n-m-j)(n+m-j)-i = k-n^2+2nj-j^2+m^2-i \equiv m \bmod 2,
		\]
		since $k \equiv n \bmod 2$ and $i \equiv j \bmod 2$. Hence $k-(n-m-j)(n+m-j)-i \in \{m, m+2, \ldots, m^2-6\}$. Since $m$ is saturated there exists a quasigroup $Q'$ of order $m$ with $\C(Q') = k-(n-m-j)(n+m-j)-i$ commuting pairs. So $\C(Q' \inc Q) = k$.
	\end{proof}
	
	We next give some explicit sufficient conditions to be able to use
	\lref{l:kconds} to conclude that $k \in \C(n)$ for some positive
	integers $n$ and some $k \in \D(n)$.
	
	\begin{lem}\label{l:kconds2}
		Let $n$ be a positive integer and let $m \leq \lfloor n/2 \rfloor$
		be a positive, saturated integer such that $4+m-m^2+2n \leq 0$. Also
		assume that $m$ is odd if $n$ is odd. If $k \in \D(n)$ is an integer
		satisfying $m+\beta(n-m) \leq k < n^2-2n-m^2+m+2$, then $k\in\C(n)$.
	\end{lem}
	
	\begin{proof}
		We show that there is some $j \in \{2, 3, \ldots, n-m\}$ such that \eref{e:k} is satisfied. Let $j \in \{2, 3, \ldots, n-m\}$ be minimal such that 
		\begin{equation}\label{e:njmb1}
			(n-j-m)(n-j+m)+\beta(j) \leq k-m.
		\end{equation}
		Setting $j=n-m$ shows that such a $j$ exists since $k \geq m+\beta(n-m)$.
		Moreover, \eref{e:njmb1} fails when $j=1$ since
		$k<n^2-2n-m^2+m+2$.
		Hence, minimality of $j$ implies that
		\begin{equation}\label{e:njmb2}
			(n-(j-1)-m)(n-(j-1)+m)+\beta(j-1) > k-m.
		\end{equation}
		It remains to show that $(n-j-m)(n-j+m)+j \geq k-m^2+6$.
		Suppose, for a contradiction, that
		$(n-j-m)(n-j+m)+j < k-m^2+6$ which together with \eref{e:njmb2} implies that
		\begin{align}
			0 &< (n-j+1-m)(n-j+1+m)+\beta(j-1)+m-((n-j-m)(n-j+m)+j+m^2-6) \nonumber\\
			&= 7+\beta(j-1)-3j+m-m^2+2n.\label{e:ineq}
		\end{align}
		It is an easy fact that $\beta(j-1)-3j \leq -3$ for all integers
		$j\geq 2$. Hence \eref{e:ineq} implies that $4+m-m^2+2n>0$, which
		contradicts the definition of $m$. The lemma now follows from
		\lref{l:kconds}.
	\end{proof}
	
	We now describe a direct construction of Latin squares which have a
	high number of commuting pairs, based on cycle switching. We need the
	following lemma on the number of commuting pairs which are destroyed
	when we switch a symmetric Latin square on a row cycle.
	
	\begin{lem}\label{l:symmswitch}
		Let $L$ be a symmetric Latin square with symbol set $S$ of
		cardinality $n$, and let $\rho$ be a row cycle of $L$ of length $k$
		which involves entries $\{i, j\} \times C$ for some
		$\{i, j\}\subseteq S$ and some $C \subseteq S$. Let $L'$ be obtained from $L$
		by switching on $\rho$. Then either $\{i, j\} \cap C = \emptyset$ or
		$\{i, j\} \subseteq C$. Furthermore,
		\begin{enumerate}[(i)]
			\item if $\{i, j\} \cap C = \emptyset$ then $L'$ has $n^2-4k$ commuting pairs,
			\item if $\{i, j\} = C$ then $L'$ is symmetric,
			\item if $\{i, j\} \subseteq C$ and $k \geq 3$ then $L'$ has $n^2-4k+6$ commuting pairs.
		\end{enumerate}
	\end{lem}
	
	\begin{proof}
		We first prove that $\{i, j\} \cap C = \emptyset$ or $\{i, j\} \subseteq C$. Suppose that $i \in C$. Then $L_{ji}$ is a symbol in an entry in $\rho$. Since $L_{ij} = L_{ji}$ it follows that $j \in C$ too. So $\{i, j\} \subseteq C$. Similarly if $j \in C$ then $\{i, j\} \subseteq C$.
		
		Let $(x, y) \in S^2$ be such that $L'_{xy} = L'_{yx}$. If $(x, y) \not\in \{i, j\} \times C$ and $(y, x) \not\in \{i, j\} \times C$ then $L'_{xy} = L_{xy} = L_{yx} = L'_{yx}$ since $L$ is symmetric. Henceforth assume that $(x, y) \in \{i, j\} \times C$. Without loss of generality $x=i$. If $(y, x) \not\in \{i, j\} \times C$ then $L'_{xy} = L'_{iy} = L_{jy}$ and $L'_{yx} = L_{yx} = L_{yi}$. Therefore $L_{jy} = L_{yi}$ which is a contradiction since $i \neq j$ and $L$ is symmetric. Finally, consider when $(y, x) \in \{i, j\} \times C$. If $x=y$ then $L'_{xy} = L'_{yx}$ is trivially true, so suppose that $y=j$. Then $L'_{xy} = L_{jj}$ and $L'_{yx} = L_{ii}$. We have shown that a pair $(x, y) \in S^2$ with $x \neq y$ is commuting if and only if $\{(x, y), (y, x)\} \cap (\{i, j\} \times C) = \emptyset$ or both $\{x, y\} = \{i, j\} \subseteq C$ and $L_{ii} = L_{jj}$.
		
		Suppose that $\{i, j\} \cap C = \emptyset$. Then $(x, y) \in S^2$ is a commuting pair if and only if $x=y$ or $\{(x, y), (y, x)\} \cap (\{i, j\} \times C) = \emptyset$. So the set of non-commuting pairs of $L$ is the set $(\{i, j\} \times C) \cup (C \times \{i, j\})$, which has cardinality $4k$.
		
		Now assume that $\{i, j\} \subseteq C$. First suppose that $k=2$. So $L_{ii} = L_{jj}$ and $L_{ij} = L_{ji}$. It follows that $L'_{ij} = L'_{ji}$ and $L'_{ii} = L'_{jj}$. Also, $L'_{xy} = L_{xy}$ for all $\{x, y\} \subseteq S$ with $\{x, y\} \neq \{i, j\}$. Thus $L'$ is symmetric. Now suppose that $k \geq 3$ so that $L_{ii} \neq L_{jj}$. Then $(x, y) \in S^2$ is a commuting pair if and only if $x=y$ or $\{(x, y), (y, x)\} \cap (\{i, j\} \times C) = \emptyset$. So the set of non-commuting pairs of $L$ is the set $(\{i, j\} \times C) \cup (C \times \{i, j\}) \setminus \{(i, i), (j, j)\}$, which has cardinality $4(k-2)+2$.
	\end{proof}
	
	Before we can make use of \lref{l:symmswitch} we need to give some definitions. 
	A \emph{Latin rectangle} is an $n \times m$ matrix, for some positive integers $n$ and $m$ with $n \leq m$, of $m$ symbols such that each symbol occurs at most once in each row and column. An $n \times m$ Latin rectangle $L$ can be \emph{completed} to a Latin square if we can add $m-n$ rows to $L$ to obtain a Latin square of order $m$. The following theorem is due to Bryant and Rodger~\cite{completesymmsquare}.
	
	\begin{thm}\label{t:completesymm}
		Let $n \geq 6$ be an integer. A $2 \times n$ Latin rectangle $R$ can
		be completed to a symmetric Latin square of order $n$ if and only if
		$R_{01}=R_{10}$ and if $n$ is odd then $R_{00}\ne R_{11}$.
	\end{thm}
	
	We can use \lref{l:symmswitch} and \tref{t:completesymm} to construct
	Latin squares with high numbers of commuting pairs.
	
	\begin{lem}\label{l:switch}
		If $n \geq 6$ is an integer, then $\{n^2-2a :3\le a\le2n-6\} \subseteq \C(n)$.
	\end{lem}
	
	\begin{proof}
		Let $k \in \{3, 4, \ldots, n-2\}$. By \tref{t:completesymm} there exists a symmetric Latin square $L_1$ of order $n$ which contains a row cycle $\rho(0, 1, 0)$ of length $k$. Let $L_1'$ be obtained from $L_1$ by switching on this row cycle. Then $L_1'$ has $n^2-4k+6$ commuting pairs by \lref{l:symmswitch}.
		
		Let $k' \in \{2, 3, \ldots, n-3\}$. By \tref{t:completesymm} there exists a symmetric Latin square $L_2$ of order $n$ which contains a row cycle $\rho = \rho(0, 1, 2)$ of length $k'$, and no entry in $\rho$ is in column $0$ or $1$ of $L_2$. Let $L_2'$ be obtained from $L_2$ by switching on $\rho$. Then $L_2'$ has $n^2-4k'$ commuting pairs by \lref{l:symmswitch}.
	\end{proof}
	
	Now that we have described our constructions of quasigroups
	and Latin squares with a known number of commuting pairs, we
	can prove \lref{l:recursive}. It will follow immediately from
	\lref{l:endn} below.  Let $n$ be a positive integer. Let
	$\E(n)$ denote the set of all integers $k$ such that, under
	the assumption that $\C(m) = \DD(m)$ for all positive integers
	$m \leq \lfloor n/2 \rfloor$, then \tref{t:anticomm},
	\cyref{c:embedding}, \lref{l:kconds} or \lref{l:switch} proves
	that $k\in \C(n)$.
	
	\begin{lem}\label{l:endn}
		Let $n \geq 10$ be an integer. Then $\E(n) = \DD(n)$.
	\end{lem}
	
	\begin{proof}
		If $n \leq 27$ then it is easy to compute 
		the set $\E(n)$ and verify that it is equal to $\DD(n)$. Henceforth
		we will assume that $n \geq 28$. Let $k \in \DD(n)$. Since a
		commutative quasigroup of order $n$ is known to exist we may assume
		that $k \neq n^2$. Define
		\[
		q = \begin{cases}
			n/2 & \text{if } n \text{ is even}, \\
			(n-1)/2 & \text{if } n \equiv 3 \bmod 4, \\
			(n-3)/2 & \text{if } n \equiv 1 \bmod 4.
		\end{cases}
		\]
		Also define
		\[
		r = \begin{cases}
			\lfloor (n-1)/3 \rfloor & \text{if } n \text{ is even or } \lfloor (n-1)/3 \rfloor \text{ is odd}, \\
			\lfloor (n-1)/3 \rfloor-1 & \text{otherwise}.
		\end{cases}
		\]
		Note that both $r$ and $q$ are odd if $n$ is odd. Also note that
		$4+q-q^2+2n \leq 0$ and $4+r-r^2+2n \leq 0$ since $n \geq 28$. So we
		can apply \lref{l:kconds2} with $m \in \{q, r\}$.
		
		Define functions $f_1 : \mathbb{Z}^2 \setminus \{(a, a-2) : a \in \mathbb{Z}\} \to \mathbb{Z}$ and $f_2 : \mathbb{Z}^2 \to \mathbb{Z}$ by
		$$f_1(a, b) = b+\lfloor (a-b)(2a-2b-3)/(a-b-2) \rfloor\text{ \ and \ }f_2(a, b) = a^2-2a-b^2+b+1.$$
		Let $x_1 = n$ and $x_2 = n+q^2-q-6$.
		Let $x_3 = f_1(n, q)$ and $x_4 = f_2(n, q)$. Let $x_5 = f_1(n, r)$ and $x_6 = f_2(n, r)$. Finally, let $x_7 = n^2-q^2+q$ and $x_8 = n^2-6$. 
		Note that
		\[
		x_2-x_3 \geq n+\left(\frac{n-3}2\right)^2-6-n
		-\frac{\left(n-\frac{n-3}2\right)(2n-(n-3)-3)}{n-\frac n2-2} = \frac{n^3-14n^2-3n+60}{4n-16},
		\]
		which is non-negative since $n \geq 28$. 
		Similarly,
		\[
		x_4-x_5\ge n^2-2n-\frac{n^2}{4}+\frac{n-3}{2}-\frac{n-1}{3}
		-\frac{(n-n/3+2)(2n-2n/3+1)}{n-(n-1)/3-2}
		=\frac{18n^3-121n^2-38n-2}{24n-60}
		\]
		and
		\[
		x_6-x_7\ge-2n-\left(\frac{n-1}{3}\right)^2+\frac n3-1
		+\left(\frac{n-3}{2}\right)^2-\frac n2
		=\frac{5n^2-124n+41}{36},
		\]
		showing that $x_5 \leq x_4$ and $x_7 \leq x_6$, respectively.
		It follows that $x_i \leq k \leq x_{i+1}$ for some $i \in \{1, 3, 5, 7\}$.	
		If $x_1 \leq k \leq x_2$ then $k \in \E(n)$ by \cyref{c:embedding}$(ii)$. If $x_3 \leq k \leq x_4$ then $k \in \E(n)$ by \lref{l:kconds2} with $m=q$. If $x_5 \leq k \leq x_6$ then $k \in \E(n)$ by \lref{l:kconds2} with $m=r$. Finally, if $x_7 \leq k \leq x_8$ then $k \in \E(n)$ by \cyref{c:embedding}$(i)$.
	\end{proof}
	
	\lref{l:recursive} follows from \lref{l:endn}. Examining the proof of
	\lref{l:endn}, it is evident that
	we used the constructions behind
	\tref{t:symmholeeven}, \tref{t:antisymmholenew} and
	\lref{l:kconds}. Other constructions, such as those in
	\lref{l:newcons}, \lref{l:symmpermutehole} and \lref{l:switch}, were
	not used, but will be used to produce base cases in the next
	subsection.
	
	\subsubsection{Base cases}
	
	In this subsection we prove the following lemma.
	
	\begin{lem}\label{l:basecases}
		For all positive integers $n \leq 10$, the statement \eref{e:goal} is true.
	\end{lem}
	
	Combining \lref{l:basecases} with \lref{l:recursive} proves \tref{t:commpairs}. Clearly, to prove \lref{l:basecases} it suffices to show that for each positive integer $n \leq 10$ and each $k \in \DD(n) \setminus \E(n)$, there is a Latin square of order $n$ which has exactly $k$ commuting pairs.
	
	Let $n \leq 10$ be a positive integer. If $n\le3$ then \eref{e:goal} is
	immediate from \tref{t:anticomm}.
	If $n=4$ then $\DD(4) \setminus \E(4) = \{6\}$.
	The following Latin square of order $4$ has $6$ commuting pairs.
	\[
	\begin{pmatrix}
		2&0&3&1\\
		1&3&2&0\\
		0&2&1&3\\
		3&1&0&2\\
	\end{pmatrix}.
	\]
	It can be verified by checking all Latin squares of order $4$ that
	none have exactly $10$ commuting pairs.
	
	If $n=5$ then $\DD(n) \setminus \E(n) = \{9, 11, 15, 19\}$.
	Let $L$ be the Latin square of order
	$5$ with $19$ commuting pairs in \eref{e:5}. For each $k \in \{9, 11,
	15\}$ we can obtain a Latin square of order $5$ with $k$ commuting
	pairs by switching $L$ on the row cycles in the table of
	\eref{e:5}. If two row cycles are listed then we switch on both
	of these cycles.
	\begin{equation}\label{e:5}
		\begin{minipage}{.45 \textwidth}
			\[
			\begin{pmatrix}
				0&1&3&4&2\\
				1&3&4&2&0\\
				3&0&2&1&4\\
				4&2&1&0&3\\
				2&4&0&3&1\\
			\end{pmatrix}
			\]
		\end{minipage}
		\begin{minipage}{.45 \textwidth}
			\[
			\begin{tabular}{|c|c|}
				\hline
				$k$&row cycles\\
				\hline
				$9$&$\rho(1, 2, 0), \rho(3, 4, 0)$\\
				$11$&$\rho(1, 2, 0), \rho(3, 4, 2)$\\
				$15$&$\rho(1, 2, 0)$\\
				\hline
			\end{tabular}
			\]
		\end{minipage}
	\end{equation}
	It can be verified by checking all Latin squares of order $5$ that none have exactly $17$ commuting pairs.
	
	If $n=6$ then $\DD(n) \setminus \E(n) = \{10, 14, 22\}$. Let $L$ be the Latin square of order $6$ with $22$ commuting pairs in \eref{e:6}. For each $k \in \{10, 14\}$ we can obtain a Latin square of order $6$ with $k$ commuting pairs by switching $L$ on the row cycles in the table of \eref{e:6}.
	\begin{equation}\label{e:6}
		\begin{minipage}{.45 \textwidth}
			\[
			\begin{pmatrix}
				1&2&3&4&5&0\\
				2&3&4&5&0&1\\
				3&4&5&0&1&2\\
				0&5&2&1&4&3\\
				4&0&1&2&3&5\\
				5&1&0&3&2&4\\
			\end{pmatrix}
			\]
		\end{minipage}
		\begin{minipage}{.45 \textwidth}
			\[
			\begin{tabular}{|c|c|}
				\hline
				$k$&row cycles\\
				\hline
				$10$&$\rho(1, 5, 0)$\\
				$14$&$\rho(1, 2, 0), \rho(4, 5, 1)$\\
				\hline
			\end{tabular}
			\]
		\end{minipage}
	\end{equation}
	
	If $n=7$ then $\DD(n) \setminus \E(n) = \{11, 17, 31\}$. Let $L$ be the Latin square of order $7$ with $31$ commuting pairs in \eref{e:7}. For each $k \in \{11, 17\}$ we can obtain a Latin square of order $7$ with $k$ commuting pairs by switching $L$ on the row cycles in the table of \eref{e:7}.
	\begin{equation}\label{e:7}
		\begin{minipage}{.45 \textwidth}
			\[
			\begin{pmatrix}
				0&1&2&3&4&5&6\\
				1&2&0&4&3&6&5\\
				2&0&1&5&6&3&4\\
				3&4&5&6&0&1&2\\
				6&3&4&0&5&2&1\\
				4&5&6&1&2&0&3\\
				5&6&3&2&1&4&0\\
			\end{pmatrix}
			\]
		\end{minipage}
		\begin{minipage}{.45 \textwidth}
			\[
			\begin{tabular}{|c|c|}
				\hline
				$k$&row cycles\\
				\hline
				$11$&$\rho(1, 2, 0), \rho(3, 5, 0)$\\
				$17$&$\rho(3, 5, 0)$\\
				\hline
			\end{tabular}
			\]
		\end{minipage}
	\end{equation}
	
	If $n=8$ then $\DD(n) \setminus \E(n) = \{16, 26, 42\}$. Let $L$ be the Latin square of order $8$ with $42$ commuting pairs in \eref{e:8}. For each $k \in \{16, 26\}$ we can obtain a Latin square of order $8$ with $k$ commuting pairs by switching $L$ on the row cycles in the table of \eref{e:8}.
	\begin{equation}\label{e:8}
		\begin{minipage}{.45 \textwidth}
			\[
			\begin{pmatrix}
				1&2&3&4&5&6&7&0\\
				2&3&4&5&6&7&0&1\\
				3&4&5&6&7&0&1&2\\
				4&5&6&7&0&1&2&3\\
				5&6&7&0&1&2&3&4\\
				0&7&2&1&4&3&6&5\\
				7&0&1&3&2&4&5&6\\
				6&1&0&2&3&5&4&7\\
			\end{pmatrix}
			\]
		\end{minipage}
		\begin{minipage}{.45 \textwidth}
			\[
			\begin{tabular}{|c|c|}
				\hline
				$k$&row cycles\\
				\hline
				$16$&$\rho(1, 2, 0), \rho(3, 5, 1)$\\
				$26$&$\rho(2, 3, 0)$\\
				\hline
			\end{tabular}
			\]
		\end{minipage}
	\end{equation}
	
	If $n=9$ then $\DD(n) \setminus \E(n) = \{17, 25, 35, 37, 47, 49, 53, 55\}$. Let $L$ be the Latin square of order $9$ with $55$ commuting pairs in \eref{e:9}. For each $k \in \{17, 25, 35, 37, 47, 49, 53\}$ we can obtain a Latin square of order $9$ with $k$ commuting pairs by switching $L$ on the row cycles in the table of \eref{e:9}. 
	\begin{equation}\label{e:9}
		\begin{minipage}{.45 \textwidth}
			\[
			\begin{pmatrix}
				1&2&3&4&5&6&7&8&0\\
				2&3&4&5&6&7&8&0&1\\
				3&4&5&6&7&8&0&1&2\\
				4&5&6&7&8&0&1&2&3\\
				5&0&7&8&3&1&2&6&4\\
				0&7&8&3&1&2&6&4&5\\
				8&6&0&1&2&3&4&5&7\\
				7&8&1&2&0&4&5&3&6\\
				6&1&2&0&4&5&3&7&8\\
			\end{pmatrix}
			\]
		\end{minipage}
		\begin{minipage}{.45 \textwidth}
			\[
			\begin{tabular}{|c|c|}
				\hline
				$k$&row cycles\\
				\hline
				$17$&$\rho(1, 2, 0), \rho(3, 4, 0)$\\
				$25$&$\rho(1, 2, 0), \rho(3, 5, 2)$\\
				$35$&$\rho(1, 3, 0)$\\
				$37$&$\rho(1, 6, 0)$\\
				$47$&$\rho(1, 8, 0)$\\
				$49$&$\rho(1, 8, 3)$\\
				$53$&$\rho(1, 8, 1)$\\
				\hline
			\end{tabular}
			\]
		\end{minipage}
	\end{equation}
	
	If $n=10$ then $\DD(n) \setminus \E(n) = \{28\}$. The order $10$ Latin square
	\[
	\begin{pmatrix}
		1&\mk2&\mk3&4&\mk5&6&7&8&9&0\\
		\mk2&3&\mk4&\mk5&\mk6&7&8&9&0&\mk1\\
		\mk3&\mk4&5&6&\mk7&8&9&0&1&2\\
		0&\mk5&2&7&4&9&6&1&8&\mk3\\
		\mk5&\mk6&\mk7&8&9&0&1&2&3&4\\
		9&0&1&2&3&4&5&6&7&8\\
		8&9&0&1&2&3&4&5&6&7\\
		7&8&9&0&1&2&3&4&5&6\\
		6&7&8&9&0&1&2&3&4&5\\
		4&\mk1&6&\mk3&8&5&0&7&2&9\\
	\end{pmatrix}
	\]
	has $28$ commuting pairs. The cells off the main diagonal corresponding to the commuting pairs have been highlighted. We have proven \lref{l:recursive} and thus also \tref{t:commpairs}.
	
	\subsection{Proof of \tref{t:Kq}}
	
	In this section we use \tref{t:commpairs} to prove \tref{t:Kq}.
	
	\begin{proof}[Proof of \tref{t:Kq}]
		Define a set $\mathcal{S}'$ by
		\[
		\mathcal{S}' = \begin{cases}
			\mathcal{S} \setminus \{4\} & \text{if } q=5/8, \\
			\mathcal{S} \setminus \{5\} & \text{if } q = 17/25, \\
			\mathcal{S} & \text{otherwise}.
		\end{cases}
		\]
		We must show that $\K(q) = \mathcal{S}'$. We first prove that $\K(q) \subseteq \mathcal{S}'$. Let $n \in \K(q)$ where $q=a/b$. Then $a/b=m/n^2$ for some $m \in \C(n) \setminus \{n^2\}$. Thus $m=a\gcd(m, n^2)$ and $n^2=b\gcd(m, n^2)$ and so $n$ is of the form $(by)^{1/2}$ for some integer $y$ such that $by$ is a square. Considering prime factorisations reveals that $y = kx^2$ for some integer $x$, where $k$ is the smallest positive integer such that $kb$ is a square. Hence $n^2=x^2kb$ for some integer $x$ and so $m=x^2ka$. Since $m \in \C(n) \setminus \{n^2\}$ we know from \lref{l:commvalid} that $m \equiv n \bmod 2$ and $n \leq m \leq n^2-6$. If $k$ is even or $a \equiv b \bmod 2$ then $x^2ka \equiv x^2kb \bmod 2$ for any $x \in \Z$. Otherwise $x^2ka \equiv x^2kb \bmod 2$ only for even $x \in \Z$. Since $m \geq n$ it follows that $x^2ka \geq x(bk)^{1/2}$ which implies that $x \geq \lceil b^{1/2}/(ak^{1/2}) \rceil$.	Since $m \leq n^2-6$ it follows that $x^2ka \leq x^2kb-6$ which implies that $x \geq \lceil 6^{1/2}/(kb-ka)^{1/2} \rceil$. We have shown that $n \in \mathcal{S}$. To conclude that $n \in \mathcal{S}'$ we note that $n \neq 4$ if $q=5/8$ and $n \neq 5$ if $q = 17/25$, by \tref{t:commpairs}.
		
		We now prove that $\mathcal{S}' \subseteq \K(q)$. Let $n \in \mathcal{S}'$. We can write $n^2 = x^2kb$ for some $x \in \Z$. Let $m = x^2ka$. The definition of the set $\mathcal{S}'$ implies that that $m \in \C(n)$ and thus $n \in \K(q)$. So $\K(q) = \mathcal{S}'$, as required.	
	\end{proof}
	
	\section*{Acknowledgements}
	
	The authors are grateful to Peter Cameron and Rosemary Bailey for
	suggesting this research topic.
	
	\printbibliography
	
\end{document}